%% file: tightness2.tex
\documentclass[10pt]{amsart}
\usepackage{hyperref}
\usepackage{amsrefs}
\usepackage{color}
\usepackage[pdftex]{graphicx}
    \newtheorem{thm}{Theorem}[section]
    \newtheorem{prop}[thm]{Proposition}
    \newtheorem{lem}[thm]{Lemma}
    \newtheorem{cor}[thm]{Corollary}

    \theoremstyle{definition}

    \newtheorem{defn}[thm]{Definition}

    \newtheorem{question}[thm]{Question}

    \theoremstyle{remark}

\newcommand{\executeiffilenewer}[3]{%
\ifnum\pdfstrcmp{\pdffilemoddate{#1}}%
{\pdffilemoddate{#2}}>0%
{\immediate\write18{#3}}\fi%
}

\newcommand{\fakeenv}{}

\newenvironment{restate}[2]                                    
{ 
 \renewcommand{\fakeenv}{#2}                              
 \theoremstyle{plain} 
 \newtheorem*{\fakeenv}{#1~\ref{#2}}                
 \begin{\fakeenv}
}
{
 \end{\fakeenv}
}

\title{Combinatorics of tight geodesics and stable lengths}

\author{Richard C. H. Webb}
\thanks{This work was supported by the Engineering and Physical Sciences Research Council Doctoral Training Grant.}
\address{Mathematics Institute, University of Warwick, Coventry, CV4 7AL, United Kingdom.}
\email{R.C.H.Webb@warwick.ac.uk}

\begin{document}

\begin{abstract}

We give an algorithm to compute the stable lengths of pseudo-Anosovs on the curve graph, answering a question of Bowditch. We also give a procedure to compute all invariant tight geodesic axes of pseudo-Anosovs.

Along the way we show that there are constants $1<a_1<a_2$ such that the minimal upper bound on `slices' of tight geodesics is bounded below and above by $a_1^{\xi(S)}$ and $a_2^{\xi(S)}$, where $\xi(S)$ is the complexity of the surface. As a consequence, we give the first computable bounds on the asymptotic dimension of curve graphs and mapping class groups.

Our techniques involve a generalization of Masur--Minsky's tight geodesics and a new class of paths on which their tightening procedure works.

\end{abstract}

\maketitle

\section{Introduction}

The curve complex was introduced by Harvey \cite{Harvey}. Its coarse geometric properties was first studied extensively by Masur--Minsky \cite{MasurMinsky99,MasurMinsky00} and it is a widely useful tool in the study of hyperbolic 3--manifolds, mapping class groups and Teichm\"uller theory.

Essential to hierarchies in the curve complex \cite{MasurMinsky00} and the Ending Lamination Theorem \cite{Minsky10,BrockCanaryMinsky12} (see also \cite{BowditchELT}) is the notion of a \textit{tight geodesic}. This has remedied the non-local compactness of the curve complex, as Masur--Minsky proved that between any pair of curves there is at least one but only finitely many tight geodesics.

Bowditch \cite{Bow08}, using hyperbolic 3--manifolds, refined the finiteness further. He also showed that the mapping class group acts \textit{acylindrically} on the curve complex. These theorems have found numerous applications in the study of mapping class groups.

In this paper we shall study tight geodesics in the curve complex from an elementary viewpoint. The techniques may be of independent interest.

\subsection{Slices of tight geodesics}

We refer the reader to Section \ref{sec:background} for the definition of the curve graph/complex, its metric and its geodesics, and Section \ref{sec:tightness} for the definition of a tight geodesic/multigeodesic.

 Let $\mathcal{L}(a,b)$ denote the set of multigeodesics that connect $a,b\in \mathcal{C}^0(S)$. Let $\mathcal{L}_T(a,b)\subset \mathcal{L}(a,b)$ denote the set of tight multigeodesics that connect $a,b\in\mathcal{C}^0(S)$. Let $G(a,b)\subset\mathcal{C}^0(S)$ be the set of curves that belong to some multicurve of some tight multigeodesic connecting $a$ to $b$.

Likewise, for subsets $A,B\subset \mathcal{C}^0(S)$ define $\mathcal{L}(A,B)$ to be the set of multigeodesics $\pi$ for which there exists $a\in A$ and $b\in B$ such that $a$ and $b$ are the endpoints of $\pi$. Define $\mathcal{L}_T(A,B)\subset \mathcal{L}(A,B)$ to be the subset of multigeodesics that are tight. Finally, $G(A,B)\subset\mathcal{C}^0(S)$ is the set of curves $c$ such that $c\in\pi$ for some $\pi\in\mathcal{L}_T(A,B)$.

Write $G(a,b;r)=G(N_r(a),N_r(b))$, and similarly $\mathcal{L}_T(a,b;r)=\mathcal{L}_T(N_r(a),N_r(b))$. The following theorem is \cite[Theorems 1.1 and 1.2]{Bow08} but with effective bounds given. The behaviour of these bounds was  not known since the only proof used geometric limiting arguments.  In the language of the following theorem, we refer to the $| G(a,b) \cap N_{\delta}(c)|$ and $|G(a,b;r)\cap N_{2\delta}(c)|$ as the `slices' of tight geodesics. (Note that curve graphs are uniformly hyperbolic \cite{Aougab13,Bow13,ClayRafiSchleimer13,HenselPrzytyckiWebb13}.)

\begin{restate}{Theorem}{slices}

Fix $\delta\geq 3$ such that $\mathcal{C}^1(S)$ is $\delta$-hyperbolic for all surfaces $S$ with $\xi(S)\geq 2$. Then the following statements hold, where $K$ is a uniform constant.

\begin{enumerate}

\item For any $a,b\in\mathcal{C}^0(S)$, for any curve $c\in\pi\in\mathcal{L}(a,b)$ we have $| G(a,b) \cap N_{\delta}(c)| \leq K^{\xi(S)}$.

\item For any $r\geq 0$, $a,b\in\mathcal{C}^0(S)$ such that $d_S(a,b)\geq 2r+2k+1$ (where $k=10\delta +1$), then for any curve $c\in\pi\in\mathcal{L}(a,b)$ such that $c\notin N_{r+k}(a)\cup N_{r+k}(b)$ we have $|G(a,b;r)\cap N_{2\delta}(c)| \leq K^{\xi(S)}$.

\end{enumerate}

\end{restate}

In Section \ref{lowerbound} we give examples to show that any upper bound in Theorem \ref{slices} has to be at least exponential in complexity. Thus the divergent behaviour of our upper bounds is best possible.

We acknowledge that prior to the work here, Shackleton has given computable bounds on $|G(a,b)|$ in terms of the intersection number of $a$ and $b$ \cite{Shack12}, as well as versions of acylindricity \cite{Shack10} in terms of intersection numbers.

\subsection{Asymptotic dimension}

This was introduced by Gromov \cite{Gromov93}. It is a quasi-isometry invariant. We say $\textnormal{asdim}X\leq n$, if for each $r>0$, $X$ can be covered by sets of uniformly bounded diameter such that each $r$-ball in $X$ intersects at most $n+1$ of these sets.

A consequence of Theorem \ref{slices}, using the work of Bell--Fujiwara \cite{BellFuji08} is that the asymptotic dimension of the curve graph can be bounded from above by an exponential of the complexity of the surface. (The curve graph also satisfies Yu's property A, see also \cite{Kida}.) Following on, Bestvina, Bromberg and Fujiwara \cite{BestBromFuji} quasi-isometrically embedded the mapping class group of a surface $S$ into a finite product of quasi-trees of curve graphs of subsurfaces of $S$, thus the asymptotic dimension of the mapping class group is finite. By Theorem \ref{slices}, it can be bounded from above by a computable function in the complexity of $S$, see \cite[Section 4.2]{BestBromFuji}. These are the first computable bounds, but are almost certainly not best possible. The exponential divergence behaviour of upper bounds on slices of tight geodesics highlights a difficulty in computing the asymptotic dimension of the curve graph. We ask

\begin{question} \label{asquestion}
Does the following hold? \begin{equation}\label{aseq} \textnormal{asdim}\mathcal{MCG}(S)=\textnormal{dim}\mathcal{EL}(S)+1 \end{equation}
\end{question}

Here the ending lamination space $\mathcal{EL}(S)$ can be $\mathcal{MCG}(S)$-equivariantly identified with the Gromov boundary of the curve graph. Gabai \cite[Chapter 15]{Gabai13} has shown that $\textnormal{dim}\mathcal{EL}(S_{0,p})=p-4$, and $\textnormal{dim}\mathcal{EL}(S_{1,p})=p-1$ (see \cite[Remark 15.11]{Gabai13}). Bell--Fujiwara \cite{BellFuji08} have calculated $\textnormal{asdim}\mathcal{MCG}(S)$ for $g\leq 2$, in particular it is $p-3$ for planar surfaces and $p$ for $S_{1,p}$. Thus \ref{aseq} holds for $g\leq 1$. Note $\mathcal{C}(S_{0,6})$ and $\mathcal{C}(S_{2,0})$ are quasi-isometric so the ending lamination spaces are homeomorphic, and (\ref{aseq}) holds again, as Bell-Fujiwara calculated $\textnormal{asdim}\mathcal{MCG}(S_{2,0})=3$.

The dimension of the ending lamination space is bounded above by the dimension of the projective filling laminations minus the dimension of the homology of the curve graph, see \cite[Chapter 15]{Gabai13}. Buyalo--Lebedeva \cite[Proposition 6.5]{BuyaloLebedeva07} have shown for hyperbolic geodesic metric spaces that $\textnormal{asdim}(X)\geq \textnormal{dim}\partial_\infty(X)+1$. For cobounded, proper, hyperbolic geodesic metric spaces the reverse inequality holds \cite{BuyaloLebedeva07}, but the curve graphs are not proper. It would be interesting to know if the equality holds for curve graphs, as this would give much better upper bounds on the asymptotic dimension of the curve graphs and mapping class groups.

A possible negative example to Question \ref{asquestion} is the surface $S_{2,1}$. It is not known if $\textnormal{dim}\mathcal{EL}(S_{2,1})\geq 4$, however, $\textnormal{asdim}\mathcal{MCG}(S_{2,1})=5$.

\subsection{Acylindrical action}

The methods in this paper provide computable constants for acylindricity (see Definition \ref{def:acy} and Theorem \ref{thm:acy}), thus, where appropriate, giving computable versions of the following theorems.

Dahmani, Guirardel and Osin \cite{DGO} have shown that for any surface (except trivial cases) there exists some $N$ such that for any pseudo-Anosov mapping class $f$, the normal closure of $f^N$ is free and purely pseudo-Anosov. Weak proper discontinuity also can be used, but here it obtains an $N$ that a priori depends on $f$. 

Fujiwara \cite{Fuji08} proved there is some $Q$ depending on the surface such that if $f$ and $g$ are pseudo-Anosovs that do not generate a virtually cyclic subgroup then there exists $M$ such that  $f^n$ and $g^m$ or $f^m$ and $g^n$ generate a free group, for all $n\geq Q$ and $m\geq M$.

Also for any surface, there is some $P$ such that for any pseudo-Anosov $f$, either $f^n$ is conjugate to its inverse with $n\leq P$, or there is a homogeneous quasimorphism $h$ on the mapping class group with $h(f)=1$ and the defect of $h$ is bounded only in terms of $S$. See \cite[Theorem C]{CaleFuji10} for more details.

The acylindrical action recovers the weak proper discontinuity condition introduced by Bestvina--Fujiwara \cite{BestFuji02} to study the bounded cohomology of non-virtually abelian subgroups of the mapping class group. They recover that the mapping class group contains no higher rank lattice. See also Hamenst\"adt \cite{Ham08} and Dru\c{t}u--Mozes--Sapir \cite{DruShaSap10} who use different methods (but use some version of acylindricity).

\subsection{Stable lengths and invariant tight geodesics}

The \textit{stable length} (or \textit{asymptotic translation length}) of a mapping class $\phi$ on the curve graph can be defined to be the value \begin{equation*} \liminf_{n\rightarrow\infty} \frac{d_S(c,\phi^n(c))}{n}.\end{equation*} Bowditch showed that stable lengths on curve graphs are uniformly rational for a fixed surface \cite{Bow08}. Here, we bound the uniform denominators of stable lengths of pseudo-Anosovs. These are the first bounds:

\begin{restate}{Theorem}{denominator}

There exists a computable $m=m(\xi(S))$ such that whenever $\phi$ is pseudo-Anosov, then the mapping class $\phi^m$ preserves a geodesic in $\mathcal{C}^1(S)$. Furthermore, $m$ is bounded by $\exp(\exp({K'\xi(S)}))$, where $K'$ is a uniform constant.

\end{restate}

Note that stable lengths of non-pseudo-Anosov mapping classes are zero. Interestingly, Gadre and Tsai \cite{GadreTsai11} have shown that the smallest non-zero stable length on the curve graph for closed surfaces $S_g$ `grows like' $\frac{1}{g^2}$. Valdivia \cite{Val13} has made appropriate statements that include the punctured surfaces, namely if $g=qp$ (where $0<q \in \mathbb{Q}$) then the minimal non-zero stable length behaves like $\chi(S)^{-2}$ and if one fixes $g$, behaves like $\chi(S)^{-1}$.

In the case of $S_{g,1}$, Zhang \cite{Zhang13a} has proved that the stable length is at least one for any point-pushing pseudo-Anosov. In particular, classifying which point-pushing pseudo-Anosovs preserve some geodesic \cite{Zhang13b}.

Lastly we provide some finite time algorithms for computing stable lengths and computing the invariant tight geodesics of any given pseudo-Anosov.

\begin{restate}{Theorem}{stablecalc}

There exists a finite time algorithm that takes as input a surface $S$, a pseudo-Anosov $\phi$ on $S$ and returns the stable length of $\phi$.

\end{restate}

\begin{restate}{Theorem}{tightcalc}

There exists a finite time algorithm that takes as input a surface $S$, a pseudo-Anosov $\phi$ on $S$ and returns all invariant tight geodesics of $\phi$. These are in the form of a collection of finite sets of curves each of whose orbit under $\phi$ is a tight geodesic.

\end{restate}

\subsection*{Acknowledgements} The author would like to thank Brian Bowditch for interesting conversations and insightful comments on early versions of this work.

\section{Background} \label{sec:background}

We write $S=S_{g,p}$ for the surface of genus $g$ with $p$ punctures and define $\xi(S)=3g+p-3$ the \textit{complexity} of $S$.

We write $\mathcal{MCG}(S)=\textnormal{Homeo}(S)/\textnormal{Homeo}_0(S)$ to denote the \textit{mapping class group}. Here, the subgroup $\textnormal{Homeo}_0(S)$ is the set of homeomorphisms homotopic to the identity homeomorphism. It is well-known that a pair of homeomorphisms that are homotopic are also isotopic for a surface $S$ with $\xi(S)\geq 1$, see for example \cite[Theorem 1.12]{FarbMargalit12}.

A \textit{curve} is an isotopy class of simple closed curve on $S$ that is \textit{essential} (does not bound a disc) and \textit{non-peripheral} (does not bound a punctured disc). We shall write $\mathcal{C}^0(S)$ for the set of curves on $S$.

In general, we say that a pair of isotopy classes of some subset on $S$ \textit{miss} if they admit disjoint representatives, and otherwise we say that they \textit{cut}.

We write $\mathcal{C}^1(S)$ for the graph with vertex set $\mathcal{C}^0(S)$ where an unordered pair of vertices $a\neq b$ share an edge if and only if $a$ misses $b$. We call $\mathcal{C}^1(S)$ the \textit{curve graph} of $S$. The \textit{curve complex} $\mathcal{C}(S)$ is the unique flag complex with 1-skeleton $\mathcal{C}^1(S)$. The curve graph is endowed with a metric where each edge has unit length. We write $d_S(a,b)$ for the distance between $a,b\in\mathcal{C}^0(S)$.

A \textit{path} is a sequence $(v_i)$ of curves $v_i\in\mathcal{C}^0(S)$ such that $v_i$ misses $v_{i+1}$ for each $i$. A path is a \textit{geodesic} if for all $i\neq j$, $d_S(v_i,v_j)=|i-j|$.

A \textit{multicurve} is a simplex of $\mathcal{C}(S)$. It is well-known that a multicurve $m$ can be realized on $S$ by a collection of pairwise disjoint and non-isotopic simple closed curves on $S$. If $c$ is a vertex of $m$, or equivalently, if $c$ is a curve of a multicurve $m$, then we write $c\in m$.

For multicurves, we define $d_S(m,m')=\min\{d_S(c,c') : c\in m, c'\in m' \}$, and generally, $d_S(A,B)=\min\{ d_S(a,b) : a\in A, b\in B\}$ for $A,B\subset \mathcal{C}^0(S)$.

Given two multicurves $m$ and $m'$, the \textit{geometric intersection number} $i(m,m')$ is defined to be the quantity $\min | \gamma\cap \gamma'|$ where the minimum is over all $\gamma$ and $\gamma'$ that intersect transversely and represent $m$ and $m'$ respectively. We say that $\gamma$ and $\gamma'$ are in \textit{minimal position} if $|\gamma\cap\gamma'|$ realizes this minimum.

Let $\alpha$ and $\beta$ both be collections of pairwise disjoint simple closed curves on $S$ and suppose that $\alpha$ and $\beta$  intersect transversely. We say that $\alpha$ and $\beta$ share a \textit{bigon} if there is a closed disc $D\subset S$ such that $\partial D$ is a union of an arc of $\alpha$ and an arc of $\beta$.

\begin{lem}
\label{big}
Let $\gamma$ and $\gamma'$ both be collections of pairwise disjoint simple closed curves on $S$, with $\gamma$ and $\gamma'$ intersecting transversely. Then:-

\begin{itemize}

\item $\gamma$ and $\gamma'$ are in minimal position if and only if $\gamma$ and $\gamma'$ do not share a bigon

\item Suppose $\gamma$ and $\gamma'$ are homotopic. Then they are ambient isotopic.

\item Suppose $\gamma$ and $\gamma'$ are disjoint and isotopic, and that no pair of simple closed curves of $\gamma$ are homotopic. Then $\gamma$ and $\gamma'$ cobound annuli, where each annulus corresponds to each simple closed curve of $\gamma$.

\end{itemize}

\end{lem}

\begin{proof} See for example \cite[Proposition 1.7]{FarbMargalit12} for the first bullet. For the second bullet, one can perturb $\gamma$ by an ambient isotopy so that the curves $\gamma$ and $\gamma'$ intersect transversely, then use the first bullet (if they intersect still) or third bullet (if they are disjoint). For the third bullet, one can lift to $\mathbb{H}^2$, for example.  \end{proof}

The following lemma states that multicurves $m$ and $m'$ that do not have a curve in common have, in a sense, unique minimal position representatives.

\begin{lem}
\label{min}
Suppose $m$ and $m'$ are multicurves that are disjoint as simplices of $\mathcal{C}(S)$. Suppose the pairs $\gamma$ and $\gamma'$ and $\alpha$ and $\alpha'$ both represent $m$ and $m'$ respectively. If $\gamma$ and $\gamma'$ is in minimal position, and likewise $\alpha$ and $\alpha'$, then as subsets of $S$, $\gamma\cup\gamma'$ and $\alpha\cup\alpha'$ are ambient isotopic.
\end{lem}

\begin{proof}

We provide a sketch. We use Lemma \ref{big} throughout. We may assume after an ambient isotopy that $\alpha'=\gamma'$ and after a small ambient isotopy of $\alpha\cup\alpha'$ preserving $\alpha'$, we make $\gamma$ and $\alpha$ intersect transversely.

If $\alpha$ and $\gamma$ intersect then they are not in minimal position so they share a bigon. Since $\alpha'$ is in minimal position with both $\gamma$ and $\alpha$, there exists an ambient isotopy that removes the bigon and preserves $\alpha'$. This reduces intersection so this process must terminate.  If $\gamma$ and $\alpha$ are disjoint then they cobound annuli, in which case there exists an ambient isotopy of $\gamma$ to $\alpha$ along these annuli, that preserves $\alpha'$ and the lemma is proved.\end{proof}

For a subset $F\subset S$ we write $n(F)$ to denote a regular neighbourhood of $F$ and $N(F)$ for the closure of $n(F)$.

We say that two multicurves $m$ and $m'$ \textit{fill} $S$ if there exist minimal position representatives $\gamma$ and $\gamma'$ of $m$ and $m'$ respectively such that $S-n(\gamma\cup \gamma')$ is a collection of discs and once-punctured discs. Note that if $m$ and $m'$ fill $S$ then any curve of $m$ must cut some curve of $m'$, therefore $m$ and $m'$ are disjoint simplices of $\mathcal{C}(S)$. It is implicit throughout that we think about subsets of $S$ up to their isotopy class so by Lemma \ref{min}, we may simply write $S-n(m\cup m')$. Note that $m$ and $m'$ fill $S$ if and only if for any curve $c$ missing $m$, $c$ must cut $m'$. Hence, if $m$ and $m'$ fill $S$ and $a$ is a multicurve that misses $m$ then $a$ and $m'$ are also disjoint simplices of $\mathcal{C}(S)$.

When $m$ and $m'$ fill $S$, we call a connected component of $S-n(m\cup m')$ a \textit{region} of $m$ and $m'$. Let $D$ be a region, and define $s(D)$ to be the number of $\textit{sides}$ of $D$, this is $2 | \partial D \cap N(m)| = |\partial D\cap N(m) | + | \partial D \cap N(m') |$. A region $D$ of $m$ and $m'$ is \textit{square} if homeomorphic to a disc and $s(D)=4$.

\begin{lem}
\label{fin}

Suppose $m$ and $m'$ fill $S=S_{g,p}$. Then $\textnormal{stab}(m)\cap \textnormal{stab}( m')\subset \mathcal{MCG}(S)$ is a finite subgroup, bounded by $12g+4p\leq 4\xi(S)+12=N_0$.

\end{lem}

\begin{proof}

Sketch. Suppose $[g]\in\textnormal{stab}(m)\cap\textnormal{stab}(m')$. Then we may take a representative $f$ of $[g]$, due to Lemma \ref{min}, such that $f$ preserves the subset $\gamma\cup \gamma'\subset S$, where $\gamma$ and $\gamma'$ represent $m$ and $m'$. By considering the regions of $\gamma$ and $\gamma'$, $[f]$ is determined by $f$ restricted to one non-square region, by extension and Alexander's trick.

If there is a punctured region with two sides (a \textit{punctured bigon}) then there are at most $p$ of these, and $f$ is determined by one's restriction, whence bounded by $2p$. If $S=S_{1,p}$ and assuming all punctured regions have four sides we can bound by $4p$ by a similar argument. For $g\geq 2$, with no punctured bigons, remove the punctures, then it suffices to bound for $S_{g,0}$. The worst case scenario is all non-square regions being \textit{hexagons} (six sides, no punctures), and there are at most $2g$ of these, whence $12g$ as a bound. \end{proof}

We say that a space $X$ is $\delta$-hyperbolic if there exists $\delta\geq 0$ such that for any points $a,b,c\in X$, and any geodesics $g_1,g_2,g_3$ connecting $a$ to $b$, $b$ to $c$ and $c$ to $a$ respectively, we have $g_1\subset N_\delta(g_2\cup g_3)$. (Here, $N_d(A)\subset X$ is the set of points that are distance at most $d$ from the set $A$.) This is called the $\delta$-slim triangles condition, and is the most convenient version of hyperbolicity for our purposes.

Note that we have defined geodesics for $\mathcal{C}^0(S)$ so it makes sense to talk about this discrete space being $\delta$-hyperbolic. Note that the curve complex and curve graph are quasi-isometric to $\mathcal{C}^0(S)$, and that hyperbolicity is a quasi-isometry invariant. We have the following crucial theorem.

\begin{thm}[Masur--Minsky \cite{MasurMinsky99}, see also \cite{Bow06},\cite{Ham07}]

The metric space $\mathcal{C}^0(S)$ is $\delta$-hyperbolic.

\end{thm}

Recently it has been proved that $\mathcal{C}^0(S)$ is uniformly hyperbolic. See for example, Aougab \cite{Aougab13}, Bowditch \cite{Bow13}, Clay--Rafi--Schleimer \cite{ClayRafiSchleimer13} and Hensel--Przytycki--Webb \cite{HenselPrzytyckiWebb13}.

\begin{thm} \label{thm:unifhyp} There exists $\delta\geq 0$ such that $\mathcal{C}^0(S)$ is $\delta$-hyperbolic when $\xi(S)\geq 2$.

\end{thm}

\section{Tightness}
\label{sec:tightness}

In this section we shall define what it means to be \textit{tight} and for a multipath to be \textit{filling}. Then we generalize Masur--Minsky's tightening procedure \cite{MasurMinsky00} to filling multipaths.

Let $m$ and $m'$ be multicurves that are disjoint simplices in $\mathcal{C}(S)$. Put $m$ and $m'$ in minimal position (well-defined due to Lemma \ref{min}). The \textit{subsurface filled by} $m$ and $m'$, written $F(m,m')$, is the closure of the union of $n(m\cup m')$ with its complementary components homeomorphic to discs or once-punctured discs. Note that $m$ and $m'$ fill $S$ if and only if $F(m,m')=S$.

A \textit{multipath} is a sequence of multicurves $(m_i)$ such that for each $i$, $m_i$ misses $m_{i+1}$. A \textit{multigeodesic} is a multipath where whenever $i\neq j$, $d_S(m_i,m_j)=|i-j|$.

\begin{defn} \label{def:fmp} A multipath $(m_i)$ is \textit{filling} if the sequence has length at least 3 and whenever $|i-j|\geq 3$ we have that $m_i$ and $m_j$ fill $S$.\end{defn}

Note that in a filling multipath, $m_i$ and $m_{i+2}$ are always disjoint simplices of $\mathcal{C}(S)$. This is important, since we will never consider the subsurface filled by $m$ and $m'$ when the multicurves share a curve. It follows from Lemma \ref{min} that $F(m_i,m_j)$ is well-defined whenever $|i-j|\geq 2$.

A filling multipath $(m_i)$ is \textit{tight} at $m_i$ (or simply $i$) if $m_i=\partial F(m_{i-1},m_{i+1})$. It is a subtle point that a priori $\partial F(m_{i-1},m_{i+1})$ had a pair of isotopic curves, so may not be a representative of $m_i$, however if we remove curves from $\partial F(m_{i-1},m_{i+1})$ until they are pairwise non-isotopic then this is a representative of $m_i$, and this is exactly what is meant by the equation above. Given a filling multipath, it makes sense to \textit{tighten} the sequence at $i$ by replacing $m_i$ with $\partial F(m_{i-1},m_{i+1})$.

\begin{lem}
\label{tight}
A tightened filling multipath is a filling multipath.
\end{lem}

\begin{proof}

Let $(m_i)$ be the original filling multipath under scrutiny, and suppose that we tighten at index $i$. Suppose for a contradiction that there exists a curve $c$ that misses $\partial F(m_{i-1},m_{i+1})$ and $m_j$, with $|i-j|\geq 3$. Without loss of generality, $j\geq i$.

Since $m_j$ and $m_i$ fill $S$, $c$ misses $m_j$ and so $c$ cuts $m_i$. The same argument for $m_{i-1}$ and $m_j$ shows that $c$ cuts $m_{i-1}$. Now, $m_i$ misses $F(m_{i-1},m_{i+1})$ and $m_{i-1}\subset F(m_{i-1},m_{i+1})$.

Our initial assumption that $c$ misses $\partial F(m_{i-1},m_{i+1})$ leads to two cases. Case 1, if $c$ misses $F(m_{i-1},m_{i+1})$ then $c$ misses $m_{i-1}$, a contradiction. Case 2, if $c$ is isotopic into $F(m_{i-1},m_{i+1})$, then $c$ misses $m_i$, a contradiction.\end{proof}

The following lemma and its proof is a generalization of Masur--Minsky \cite[Lemma 4.5]{MasurMinsky00}. In their work, they proved the analogous statement but for geodesics. For clarity and completeness to our generalization, we include the proof here.

\begin{lem}
\label{proc}
Suppose $(m_i)$ is a filling multipath that is tight at $i$. After tightening at $i+1$, the resulting filling multipath is tight at $i$.
\end{lem}

\begin{proof}

Let $m'_{i+1}=\partial F(m_i,m_{i+2})$. We argue that $F(m_{i-1},m_{i+1})=F(m_{i-1},m'_{i+1})$.

Firstly, we argue that $m_{i+1}'\subset F(m_{i-1},m_{i+1})$ to deduce that $F(m_{i-1},m_{i+1}')\subset F(m_{i-1},m_{i+1})$. The multicurve $m_{i+1}'$ misses $\partial F(m_{i-1},m_{i+1})=m_i$. Since $m_{i+1}'$ misses $m_{i+2}$, we have that each curve of $m_{i+1}'$ cuts $m_{i-1}$ and therefore $m_{i+1}'\subset F(m_{i-1},m_{i+1})$ as required.

For the other inclusion we argue by contradiction. Suppose there is a curve $c$ that is essential and non-peripheral in $F(m_{i-1},m_{i+1})$ but misses $F(m_{i-1},m'_{i+1})$. So $c$ misses $m_{i-1}$ and it follows that $c$ cuts $m_{i+1}$. Also, $c$ must cut $m_{i+2}$ because $m_{i-1},m_{i+2}$ fill $S$.

Summarizing, $c$ misses $m'_{i+1}$ so there are two cases. Case 1, $c$ misses $F(m_i,m_{i+2})$, contradicting $c$ cuts $m_{i+2}$. Case 2, $c$ is isotopic into $F(m_i,m_{i+2})$, contradicting $c$ cuts $m_{i+1}$. \end{proof}

Let $P=(m_0,...,m_L)$ be a filling multipath or multigeodesic. We say $P$ is \textit{tight} if for each $i$ with $0<i<L$, $m_i=\partial F(m_{i-1},m_{i+1})$.

Note that any curve on some Bowditch tight geodesic \cite{Bow08} is a curve on some Masur--Minsky tight geodesic---this follows from the \textit{tightening procedure} (this is a combination of Lemmas \ref{tight} and \ref{proc}), starting with the curve under scrutiny. Thus, Theorem \ref{slices} applies to either setting.

\section{Filling multiarcs}

The goal of this section is to provide a bound on the number of tight filling multipaths connecting two vertices, purely in terms of $S$ and the lengths of the multipaths, and not in terms of the endpoints. Here, we make no use of hyperbolicity and the arguments are combinatorial.

Write $\mathcal{MC}^0(S)$ for the set of multicurves of $S$. Analogous to $\mathcal{C}^1(S)$, define a graph $\mathcal{MC}(S)$ with vertex set $\mathcal{MC}^0(S)$, where an unordered pair of vertices are adjacent if the multicurves miss. Make every edge have unit length and we write $d_{\mathcal{MC}(S)}$ for the induced metric. Note that $d_{\mathcal{MC}(S)}(m,m')\geq 3$ if and only if $m,m'$ fill $S$.

An \textit{arc} of a surface with boundary $S'$ is an isotopy class (with endpoints on the boundary throughout the isotopy) of a proper embedding of the compact interval in $S'$, which is not isotopic into the boundary. Analogous to the curve complex, one can define the \textit{arc complex} $\mathcal{A}(S')$. A \textit{multiarc} is a simplex in the complex $\mathcal{A}(S')$.

Given $s\in\mathcal{MC}^0(S)$, write $\mathcal{MA}^0(S,s)$ for the set of multiarcs of $S-n(s)$ with endpoints on $\partial N(s)$. Define $\mathcal{MA}(S,s)$ to be the graph with vertex set $\mathcal{MA}^0(S,s)$ and two non-equal multiarcs share an edge if they miss.

Write $\mathcal{FMA}^0(S,s)$ for the set of multiarcs $m\in\mathcal{MA}^0(S,s)$ such that $(S-n(s))-n(m)$ is a union of discs and once-punctured discs. We remark that $(S-n(s))-n(m)$ is well-defined since multiarcs are homotopic if and only if they are ambient isotopic. Such a multiarc $m$ is called a \textit{filling multiarc of} $(S,s)$.

Define $\mathcal{FMA}(S,s)$ to be the graph with vertex set $\mathcal{FMA}^0(S,s)$, and edges span a pair of vertices $m,m'$ if and only if $m$ misses $m'$.

\begin{lem}
\label{A}
Let $s\in\mathcal{MC}^0(S)$ then $\mathcal{FMA}(S,s)$ has degree bounded by $d_0=2^{24\xi(S)+108}$.
\end{lem}

\begin{proof}

Let $m$ be a filling multiarc. For each region $D$ of $m$, we shall bound the number of multiarcs adjacent to $m$ contained in $D$, in terms of $s(D)$. Then we bound the sums of $s(D)$ in terms of $\xi(S)$.

Any multiarc in a punctured region $D$ with $s=s(D)$ sides is a subset of a maximal multiarc, and this envelopes the puncture of $D$ in a punctured bigon, leaving outside a region of $s+2$ sides and a maximal multiarc there. The number of maximal multiarcs of a non-punctured region corresponds to the Catalan numbers $C_n$, where $C_0$=1 and $C_n=2(2n-1)C_{n-1} / (n+1)\leq 4^n$. The number of full triangulations of a convex polygon with $n+2$ sides is equal to $C_n$. Thus the number of maximal multiarcs in any given region (punctured or non-punctured) is at most $\frac{s}{2}2^{s-2}\leq 2^{2s-2}$. Now any multiarc is a subset of a maximal multiarc and therefore the number of multiarcs in a given region $D$ with $s$ sides is at most $2^{2s-2}.2^{s}=2^{3s-2}\leq 2^{3s}$.

The degree of $m$ is bounded by the product of possibilities of multiarc in each region, so it suffices to bound the sum $\Sigma s(D)$, over all regions $D$ of $m$, in terms of $\xi(S)$. Let $N$ be the number of regions of $m$. By gluing along the arcs of $m$, capping off punctures and considering Euler characteristic, $\Sigma s(D)$ is equal to $4N-4\chi (S_g)$. So it suffices to bound $N$. By taking a maximal filling multiarc, the worst case scenario is where all regions are either punctured bigons ($n$ of them) or non-punctured hexagons ($h$ of them). We see that $\chi(S_g)=(n+h)-(\frac{n}{2}+\frac{3h}{2})$ thus $h=n-2\chi(S_g)$ thus $N\leq 2n-2\chi(S_g)$ and so $3\Sigma s(D)\leq 3(8n-12\chi(S_g))\leq 24\xi(S)+108$.\end{proof}

Let $s\in\mathcal{MC}^0(S)$. Write $\mathcal{MC}(S;s,3)=\{ s'\in\mathcal{MC}^0(S) : d_{\mathcal{MC}(S)}(s,s')\geq 3 \}$. Define $\kappa_s: \mathcal{MC}(S;s,3)\rightarrow \mathcal{FMA}(S,s)$ the \textit{cutting map}, where $\kappa_s(s')$ is the simplex of arcs $s'-n(s)\subset S-n(s)$, well-defined due to Lemma \ref{min}.

We write $N_i(A)$ for the set of points at most distance $i$ from the set $A$. Define $\mathcal{MC}(S;s,2)=N_1(\mathcal{MC}(S;s,3))$, here the notion of distance for $N_1$ is in $\mathcal{MC}(S)$. This set has an inclusion into the multicurves of distance at least 2 from $s$, but generally it is not an equality, because one may have distance 2 multicurves (in $\mathcal{MC}(S)$) that share a curve in common. Any curve of $m\in\mathcal{MC}(S;s,2)$ must cut $s$. Notice that $\kappa_s$ extends to a map $\mathcal{MC}(S;s,2)\rightarrow\mathcal{MA}^0(S,s)$, for which we write $\kappa_s$ also, by Lemma \ref{min}.

Define a relation $p_s:\mathcal{FMA}(S,s)\rightarrow \mathcal{MA}(S,s)$ where $m\mapsto m'$ if and only if $m'$ is a sub-multiarc of $m$.

\begin{lem}
\label{1}
Given $s\in\mathcal{MC}^0(S)$ and $s_0,s_1,...,s_l$ a multipath with $s_i\in\mathcal{MC}(S;s,3)$ for each $i$, then $\kappa_s(s_0)\in N_l(\kappa_s(s_l))\subset \mathcal{FMA}(S,s)$ and this set has a computable bound on its cardinality in terms of $l$ and $\xi(S)$.
\end{lem}

\begin{proof}

It suffices to show that $\kappa_s(s_i)$ misses $\kappa_s(s_{i+1})$, and this is clear. Then apply Lemma \ref{A} to deduce the cardinality is bounded by $d_0^{l+1}$.\end{proof}

\begin{lem}
\label{2}
Let $s,s',s''\in\mathcal{MC}^0(S)$, suppose $s'$ misses $s''$ and that $s$ and $s''$ fill $S$. Then $\kappa_s(s')\in p_s(N_1(\kappa_s(s'')))$ and the latter set has at most $d_0$ multiarcs.
\end{lem}

\begin{proof}

Note that $\kappa_s(s')\cup \kappa_s(s'')$ is a filling multiarc of $(S,s)$ that misses $\kappa_s(s'')$, hence $\kappa_s(s')$ is a sub-multiarc of $\kappa_s(s')\cup\kappa_s(s'')\in N_1(\kappa_s(s''))$. The number of possibilities, following the proof of Lemma \ref{A}, is at most $d_0$. \end{proof}

Let $s\in\mathcal{MC}^0(S)$ and $m\in\mathcal{MA}(S,s)$. Define $F(s;m)\subset S$ to be the resulting subsurface when taking a regular neighbourhood $n(s\cup m)$ union with complementary discs and once-punctured discs then taking the closure.

Define a relation $T_s:\mathcal{MA}(S,s)\rightarrow \mathcal{MC}^0(S)$ by $m\mapsto \partial F(s;m)$. Here, a filling multiarc would relate to the empty set.

\begin{lem}
\label{B}
Suppose $s\in\mathcal{MC}^0(S)$ and $s'\in\mathcal{MC}(S;s,2)$ then $F(s; \kappa_s(s'))=F(s,s')$.
\end{lem}
\begin{proof}Firstly, $s$ and $s'$ are disjoint simplices of $\mathcal{C}(S)$. It suffices to show that after adding some square regions to $n(s\cup s')$, one has a regular neighbourhood $n(s\cup\kappa_s(s'))$.

It is a fact that two disjoint embedded arcs are isotopic if and only if they cobound a square with the boundary of the surface. One can see this by doubling the surface then applying a generalization of the last part of Lemma \ref{big}.

Let $\gamma,\gamma'$ be a choice of representatives of $s,s'$ in minimal position. We consider the arcs $\gamma'-n(\gamma)\subset S-n(\gamma)$. Let $A_1,...,A_n$ be the partition of connected components (i.e. the arcs) of $\gamma'-n(\gamma)$ into their isotopy classes in $S-n(\gamma)$, hence an arc in $A_i$ is isotopic to an arc in $A_j$ if and only if $i\neq j$.

Pick an arbitrary choice $\alpha_i$ of arc in $A_i$, for each $i$. Thus, $\bigcup \alpha_i$ is a representative of $\kappa_s(s')$.

Now consider $n(\gamma\cup\gamma')$. Since isotopic arcs cobound squares, as mentioned above, if $A_i$ consists of $n$ arcs then there are $n-1$ squares inbetween the collection of arcs $A_i$. Thus after adding these squares to $n(\gamma\cup\gamma')$, one has a regular neighbourhood $n(\gamma\cup\bigcup\alpha_i)$. Hence we have shown $n(s\cup\kappa_s(s'))$ has the structure mentioned above, as required. \end{proof}

\begin{cor}
\label{3}
Suppose $s_0,s_1,s_2,s_3$ is a filling multipath that is tight at $s_1$. Then $s_1=T_{s_0}(\kappa_{s_0}(s_2))$. \hfill $\square$
\end{cor}

\begin{lem}
\label{I}
Suppose $a=s_0,s_1,...,s_l=b$ is a tight filling multipath. If $i$ is an integer $0\leq i\leq l-3$ then $s_{i+1}\in T_{s_i}(p_{s_i}(N_{l-i-2}(\kappa_{s_i}(b))))$. The cardinality of $T_{s_i}(p_{s_i}(N_{l-i-2}(\kappa_{s_i}(b))))$ as a set of multicurves is bounded by $d_0^{l-i-1}$.
\end{lem}

\begin{proof}

For each $j$ such that $i+3\leq j \leq l$, $\kappa_{s_i}(s_j)\in\mathcal{FMA}^0(S,s_i)$. By Lemma \ref{1}, $\kappa_{s_i}(s_{i+3})\in N_{l-i-3}(\kappa_{s_i}(b))$. By Lemma \ref{2}, $\kappa_{s_i}(s_{i+2})\in p_{s_i}(N_1(\kappa_{s_i}(s_{i+3})))\subset p_{s_i}(N_{l-i-2}(\kappa_{s_i}(b)))$ and Corollary \ref{3} finishes the proof.\end{proof}

Given a subset $M\subset\mathcal{MC}^0(S)$, we define the set $I_j(M,b)=\bigcup \{ T_m(p_m(N_j(\kappa_m(b)))) : m\in M\cap \mathcal{MC}(S;b,3) \}$. One can think of this set consisting of neighbours of $M$ towards $b$. We define $T(M,b)=\bigcup \{ \partial F(m;b) : m\in M \cap \mathcal{MC}(S;b,2) \}$. This set consists of the tightenings of multicurves of $M$ with $b$. We can rephrase Lemma \ref{I} to $s_{i+1}\in I_{l-i-2}(s_i,b)$.

We define $C(a,b;L)\subset \mathcal{C}^0(S)$ to be the set of curves $c$ such that $c$ lies on some tight filling multipath $P$ connecting $a$ to $b$ with $\textnormal{length}(P)\leq L$. Note that when $L\geq d_S(a,b)\geq 3$, the tight geodesics from $a$ to $b$ lie in this set. The following theorem, then, is a generalization of Masur--Minsky: that there are only finitely many tight geodesics between two given curves. Furthermore, this is the first computable bound in terms of the distance of the curves and the surface, improving on Shackleton \cite{Shack12}.

\begin{thm}
\label{thm:tight}

Given $S$ and an integer $L\geq 0$ there exists $B=B(S,L)$ such that given any $a,b\in\mathcal{C}^0(S)$ we have $|C(a,b;L)|\leq B \leq 2\xi(S) L d_0^{L^2}$.

\end{thm}

\begin{proof}

The theorem is vacuous if $a$ and $b$ do not fill $S$. So suppose $a$ and $b$ fill $S$. Define $C_0=\{ a\}$ and $C_i=C_{i-1}\cup T(C_{i-1},b)\cup I_{L-i-1}(C_{i-1},b)$. We take $C=C_{L-1}$. Note that $C_i\subset C_{i+1}$. We claim that $C(a,b;L)\subset C$.

Given any $a=s_0,s_1,...,s_k=b$ a tight filling multipath with $k\leq L$, by Lemma \ref{I} and induction, for all integers $i$ with $0\leq i \leq k-3$ we have $s_{i+1}\in C_{i+1}$. Also, $s_{k-1}=\partial F(s_{k-2},s_k)\in T(C_{k-2},b)\subset C_{k-1}$.

Now we bound $|C|$. We first bound the multicurves from the $I_{L-i-1}(C_{i-1},b)$, forgetting the $T(C_{i-1},b)$ for the moment. By Lemma \ref{I} inductively, this is at most $1+d_0^{L-1}+d_0^{L-1+L-2}+...+d_0^{L-1+L-2+...+2}\leq L d_0^{L^2}$. Now we take care of the $T(C_{i-1},b)$. We note that this is a terminating map, finishing off the tight filling multipath from any of the $I_{L-i-1}(C_{i-1},b)$, thus it is bounded by $2L d_0^{L^2}$. Finally given any multicurve, there are at most $\xi(S)$ components, therefore $B\leq 2\xi(S) L d_0^{L^2}$.\end{proof}

Note that in the case of tight geodesics, a bounded geodesic image theorem holds for $\kappa_s$ for when a geodesic vertex wise fills with $s$ \cite{Webb13}. Therefore some improvements can be made on the above bound $B$ in the case of tight geodesics.

We remark that the proof of Theorem \ref{thm:tight} gives an algorithm to compute the distance of a given pair of curves $a$ and $b$. We can also compute all tight geodesics from $a$ to $b$. These results are not new, see for example Leasure \cite{Leasure} and Shackleton \cite{Shack12}. We hope to return to this subject in a later paper, as the procedure for an algorithm implied in Theorem \ref{thm:tight} may have an interesting running time.

\section{Constructing filling multipaths}
\label{sec:construct}

The goal in this section is to give a procedure for constructing filling multipaths from multigeodesics, and tight filling multipaths from tight multigeodesics.

Some of the work here in this section was inspired by Shackleton \cite{Shack10}. This section makes no use of hyperbolicity and the arguments are elementary.

We say $m_i,m_j$ \textit{fail to fill} if $(m_i)$ is a multipath of length at least 3 and there exists $j-i\geq 3$ such that $m_i,m_j$ do not fill.

We start with an easy lemma.

\begin{lem}
\label{geod}
If $m_0,...,m_n$ is a multipath and $d_S(m_0,m_n)=n$ then $m_0,...,m_n$ is a multigeodesic. \hfill $\square$
\end{lem}

We write $m_n,m_{n+1},...$ to mean the rest of a multipath or multigeodesic. It may or may not terminate.

\begin{lem}
\label{fill}
If $m_0,...,m_n$ and $m_{n-2},m_{n-1},m_n,...$ are multigeodesics and $d_S(m_0,m_j)\geq n$ whenever $j\geq n$ then $m_0,m_1,...$ is a filling multipath.
\end{lem}

\begin{proof}
Suppose for contradiction that $m_i,m_j$ fail to fill. Then by multigeodesics, $i\leq n-3$ and $j\geq n+1$ but this implies that $d_S(m_0,m_j) \leq d_S(m_0,m_i)+2 =i+2 \leq n-1$ whereas $j\geq n+1$ and so $d_S(m_0,m_j)\geq n$ a contradiction.\end{proof}

Sometimes we can pinpoint where a multipath fails to fill with the following lemma.

\begin{lem}
\label{place}
Suppose $m_0,...,m_n$ and $m_n,m_{n+1},m_{n+2},m_{n+3},...$ are multigeodesics and $d_S(m_0,m_i)\geq n+1$ whenever $i\geq n+1$. If $m_0,...,m_n,...$ is not filling then only $m_{n-1},m_{n+2}$ or $m_{n-1},m_{n+3}$ can fail to fill.
\end{lem}

\begin{proof}

If $m_i,m_j$ fail to fill then by multigeodesics, $i\leq n-1$ and $j\geq n+1$. Suppose for contradiction that $i\leq n-2$ then $d_S(m_0,m_j)\leq d_S(m_0,m_i)+2=i+2\leq n$ whereas $d_S(m_0,m_j)\geq n+1$. So $i=n-1$. Now, there exists $c$ such that $m_{n-1},c,m_j$ is a multipath but $m_n,m_{n-1},c,m_j$ is a multipath of length 3 hence $j\leq n+3$. So $j=n+2$ or $n+3$ because $j-i\geq 3$.\end{proof}

The following is the most technical in this section, but it is key.

\begin{figure}
\centering
\includegraphics{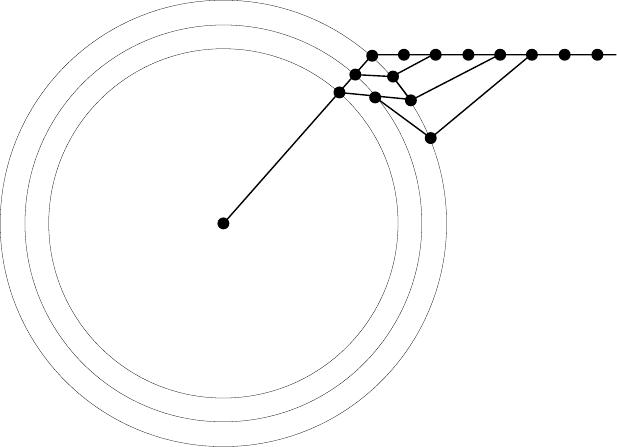}
\caption{The argument for Proposition \ref{con} going through the worst case scenario: $j=r+2,j'=r+4,j''=r+5$.}
\label{fig:conn}
\end{figure}

\begin{prop}
\label{con}
Suppose that $m_0,...,m_r$ and $m_r,m_{r+1},...,m_{r+5},...$ are multigeodesics and that $d_S(m_0,m_i)\geq r+1$ whenever $i\geq r+1$. Then there is a filling multipath $P$ starting at $m_0$, with sub-multipath $m_{r+5},...$, and the length of the multipath from $m_0$ to $m_{r+5}$ is at most $r+5$.
\end{prop}

\begin{proof}

The proof is by exhaustion. See Figure \ref{fig:conn}.

Write $P_0=m_0,...,m_r,m_{r+1},...$. By Lemma \ref{place}, with $n=r$, if $P_0$ is not filling then $m_{r-1},m_j$ fails to fill with $j=r+2$ or $r+3$. We take $j$ to be maximal. So there exists $c_1$ with $m_{r-1},c_1,m_j$ a multipath. It follows that $m_0,...,m_{r-1},c_1,m_j$ is a multigeodesic of length $r+1$ by Lemma \ref{geod}. Write $P_1=m_0,...,m_{r-1},c_1,m_j,...$. If $j=r+3$ then $m_r,m_{r-1},c_1,m_j,m_{j+1},...$ is a multigeodesic. Hence by Lemma \ref{fill}, with $n=r+1$, $P_1$ is a filling multipath.

The other case is if $j=r+2$. If $P_1$ is not filling then $c_1,m_{j'}$ fails to fill, $j'\geq r+4$. We took $j$ to be maximal earlier, so it cannot be the case that $m_{r-1}$ fails to fill with some other multicurve in $P_1$.

So there exists $c_2$ with $c_1,c_2,m_{j'}$ a multipath but $m_r,m_{r-1},c_1,c_2,m_{j'}$ is a multipath of length 4 hence $j'\leq r+4$ and so $j'=r+4$. It follows that $m_{r-1},c_1,c_2,m_{r+4},m_{r+5},...$ is a multigeodesic. Write $P_2=m_0,...,m_{r-1},c_1,c_2,m_{r+4},m_{r+5},...$.

If $d_S(m_0,c_2)=r+1$ then by Lemma \ref{geod}, $m_0,...,m_{r-1},c_1,c_2$ is a multigeodesic and by Lemma \ref{fill} with $n=r+1$ $P_2$ is filling.

The other case is if $d_S(m_0,c_2)=r$. By Lemma \ref{place} with $n=r-1$, if $P_2$ is not filling then only $m_{r-2},c_2$ or $m_{r-2},m_{r+4}$ can fail to fill. But $d_S(m_0,m_{r+4})\geq r+1$ and $d_S(m_0,m_{r-2})=r-2$ so $m_{r-2},m_{r+4}$ fill. So there exists $c_3$ with $m_{r-2},c_3,c_2$ a multipath. Write $P_3=m_0,...,m_{r-2},c_3,c_2,m_{r+4},m_{r+5},...$. By Lemma \ref{geod}, $m_0,...,m_{r-2},c_3,c_2,m_{r+4}$ is a multigeodesic of length $r+1$.

If $P_3$ fails to fill then by Lemma \ref{place}, with $n=r$, it must fail to fill with $c_3,m_{j''}$, $j''=r+5$ or $r+6$. So there exists $c_4$ such that $c_3,c_4,m_{j''}$ is a multipath. But $m_r,m_{r-1},m_{r-2},c_3,c_4,m_{j''}$ is a length $5$ multipath so $j''\leq r+5$ and so $j''=r+5$. It follows that $m_{r-2},c_3,c_4,m_{r+5},...$ is a multigeodesic. Also, $m_0,...,m_{r-2},c_3,c_4$ is a multigeodesic of length $r$ by Lemma \ref{geod}. By Lemma \ref{fill}, with $n=r$, $P_4=m_0,...,m_{r-2},c_3,c_2,c_4,m_{r+5},...$ is a filling multipath.\end{proof}

\subsection{Constructing tight filling multipaths}

We now aim to use Proposition \ref{con} and the following two straightforward lemmas to give a recipe for constructing tight filling multipaths (Theorem \ref{dir}). This is crucial for our proofs in Section \ref{sec:slice}.

\begin{lem}
\label{order}

Let $m=(m_i)$ be a multigeodesic, $r\geq 0$ and suppose there are vertices $c_a,c,c_b\in \mathcal{C}^0(S)$ such that $d_S(c',m)\leq r$ whenever $c'\in \{ c_a,c,c_b\}$, $d_S(c_a,c)+d_S(c,c_b)=d_S(c_a,c_b)$ and $d_S(c_a,c)=d_S(c,c_b)=4r+1$. Then we may reorder the indices of $m$ such that whenever $d_S(c_a,m_i),d_S(c,m_j),d_S(c_b,m_k)\leq r$ then $i<j<k$.

\end{lem}

\begin{proof}

Firstly, $|j-i|=d_S(m_i,m_j) \leq r+ 4r+1 +r =6r+1$, similarly $|k-j| \leq 6r+1$, and $|k-i|\leq 10r+2$. Also, $4r+1 =d_S(c_a,c)\leq r+|j-i|+r$ so $2r+1\leq |j-i|$. Similarly, $2r+1\leq |k-j|$ and $8r+2\leq r+ |k-i| +r$ so $6r+2\leq |k-i|$. It is now clear that $i,j,k$ are distinct.

Suppose $i<k<j$. Then $6r+2\leq k-i \leq j-i \leq 6r+1$, a contradiction. This argument also rules out the cases $j<k<i$, $j<i<k$ and $k<i<j$.

Suppose $i<j<k$ for some $i,j,k$. (Note that if this is not the case, then reorder $m$ and the lemma is proved.) Suppose for contradiction that there exist $i'>j'>k'$, with $d_S(c_a,m_{i'})\leq r$, etc. By $m$ multigeodesic, $|i'-i|\leq 2r$, and as above, $|i'-j|\geq 2r+1$, so $i'<j$ and $j-i'\geq 2r+1$. Similarly, we must have $|j-j'|\leq 2r$, but $j-j' = (j-i') +( i' -j') \geq 2r+1+1$, a contradiction. Thus we must have $i'<j'$ and hence $i'<j'<k'$.\end{proof}

\begin{lem}
\label{where}

Let $m$ be a multigeodesic, $r\geq 3$ and $c_a,c,c_b\in \mathcal{C}^0(S)$ be such that $d_S(c',m)\leq r$ whenever $c'\in \{ c_a,c,c_b\}$, $d_S(c_a,c)+d_S(c,c_b)=d_S(c_a,c_b)$ and $d_S(c_a,c)=4r+1$. Assume further that $m$ is ordered with respect to Lemma \ref{order}. Then $m\cap N_r(c)\subset m_{i+6}\cup ... \cup m_{k-6} \subset \mathcal{C}^0(S)$, where $i$ is largest such that $d_S(m_i,c_a)\leq r$ and $k$ is smallest such that $d_S(m_k,c_b)\leq r$.

\end{lem}

\begin{proof}
Suppose $v\in m\cap N_r(c)$ then $v\in m_j$ for some $j$. By Lemma \ref{order}, $i<j<k$. Also, $j-i\geq 2r+1\geq 6$ and similarly $k-j\geq 6$. Thus $j\in \{i+6,...,k-6\}$.\end{proof}

\begin{thm}
\label{dir}

Let $\pi$ be a tight multigeodesic and $r\geq 3$. Then the following statements hold:

\begin{enumerate}

\item If there exists $c_a,c,c_b\in \mathcal{C}^0(S)$ such that $d_S(c',\pi)\leq r$ whenever $c'\in\{ c_a,c,c_b\}$, $d_S(c_a,c)+d_S(c,c_b)=d_S(c_a,c_b)$ and $d_S(c_a,c)=d_S(c,c_b)=4r+1$, then there exists a tight filling multipath $P$ from $c_a$ to $c_b$ such that $\textnormal{length}(P)\leq 12r+2$ and $\pi\cap N_r(c)\subset P\cap N_r(c)$.

\item Suppose $b$ is an endpoint of $\pi$ and that there exist $c_a,c\in \mathcal{C}^0(S)$ such that $d_S(c',\pi)\leq r$ whenever $c'\in\{ c_a,c \}$, $d_S(c_a,c)+d_S(c,b)=d_S(c_a,b)$, $d_S(c_a,c)=4r+1$ and $d_S(c,b)\leq 4r+1$. Then there exists a tight filling multipath $P$ from $c_a$ to $b$ such that $\textnormal{length}(P)\leq 12r+2$ and $\pi\cap N_r(c)\subset P \cap N_r(c)$.

\end{enumerate}

\end{thm}

\begin{proof}

For brevity we shall only write out the proof of (1), since the proof of (2) is analogous.

Use Lemma \ref{order} to order $\pi$ and define $i$ to be largest such that $d_S(\pi_i, c_a)\leq r$ and $k$ to be smallest such that $d_S(\pi_k,c_b)\leq r$. We have $6r+2 \leq k-i\leq 10r+2$.

By Lemma \ref{where}, we have $\pi\cap N_r(c)\subset \pi_{i+6}\cup ...\cup \pi_{k-6}$. Thus it suffices for the tight filling multipath $P$ to contain these multicurves.

Let $c'_a\in \pi_i$ and $c'_b\in\pi_k$ be curves. Let $g_a$ ($g_b$) be a geodesic connecting $c_a$ to $c'_a$ ($c'_b$ to $c_b$). Let $m$ be the multigeodesic such that $m_i=c'_a$, $m_k=c'_b$, and whenever $i<j<k$, $m_j=\pi_j$. Using Proposition \ref{con} with the concatenation of $g_a$ and $m$ we obtain a filling multipath $P_1$ from $c_a$ to $m_k$ with submultipath $m_{i+5},...,m_k$ and whose length from $c_a$ to $m_{i+5}$ is at most $r+5$. Similarly, using Proposition \ref{con} with $m$ and $g_b$, we obtain $P_2$ a filling multipath from $m_i$ to $c_b$ with submultipath $m_i,...,m_{k-5}$ and whose length from $m_{k-5}$ to $c_b$ is at most $r+5$.

Now we construct $P_3$ a multipath from $c_a$ to $c_b$ of length at most $12r+2$, by concatenating $P_1$ from $c_a$ to $m_{k-5}$, and $P_2$ from $m_{k-5}$ to $c_b$. This multipath has submultipath $m'=m_{i+5},...,m_{k-5}$. Now, $P_3$ is filling: suppose for contradiction that $(P_3)_p,(P_3)_q$ fail to fill. Then since $P_3$ is contained in the union of $P_1$ and $P_2$, we must have $(P_3)_p \notin P_2$ and $(P_3)_q\notin P_1$, and therefore since $m'\subset P_1\cap P_2$, $(P_3)_p, (P_3)_q\notin m'$. Therefore, $d_S((P_3)_p, c_a)\leq r+4$ and $d_S((P_3)_q,c_b)\leq r+4$, but failing to fill implies $d_S((P_3)_p,(P_3)_q)\leq 2$ and hence $d_S(c_a,c_b)\leq 4r+10 < 8r < 8r+2$ a contradiction.

Now $P_3$ has submultipath $m_{i+5},...,m_{k-5}$, and this is $\pi_{i+5},...,\pi_{k-5}$, which is tight everywhere at $\pi_{i+6},...,\pi_{k-6}$. Now we use Lemmas \ref{tight} and \ref{proc} to tighten $P_3$ at every vertex to gain a tight filling multipath $P_4$ of length at most $12r+2$ from $c_a$ to $c_b$. Furthermore, $P_4$ has submultipath $\pi_{i+6},...,\pi_{k-6}$ by Lemma \ref{proc}, and hence $\pi\cap N_r(c)\subset P_4$. We set $P=P_4$ and the theorem is proved. \end{proof}

\section{Proofs and bounds}

\subsection{Upper bounds on slices}

\label{sec:slice}

Recall the terminology from the introduction.

\begin{lem}
\label{fellow}

Given $r\geq 0$, $a,b\in\mathcal{C}^0(S)$, if $c\in\pi\in\mathcal{L}(a,b)$ such that $d_S(c,\{ a,b \})\geq r+2\delta+1$ then any geodesic connecting $N_r(a)$ to $N_r(b)$ must intersect $N_{2\delta}(c)$.

\end{lem}

\begin{proof}

Suppose $\pi'$ is a geodesic that connects $a',b'$ with $d_S(a,a'),d_S(b,b')\leq r$. It suffiices to show that $\pi'\cap N_{2\delta}(c)\neq \emptyset$.

There exists a geodesic $\pi_a$ connecting $a$ to $a'$, a geodesic $\pi_b$ connecting $b$ to $b'$ and a geodesic $\Pi$ connecting $a$ to $b'$.

By the slim triangles condition, $c$ is contained in the $\delta$-neighbourhood of $\pi_b\cup\Pi$, using the triangle $a,b,b'$. But $r+2\delta+1\leq d_S(b,c)\leq d_S(b,\pi_b)+d_S(\pi_b,c)\leq r+d_S(\pi_b,c)$ so $2\delta+1\leq d_S(\pi_b,c)$, so there exists $c'\in \Pi$ such that $d_S(c,c')\leq \delta$ and $d_S(\{ a,b \},c')\geq r+\delta+1$. A similar argument again shows that there exists $c''\in \pi'$ such that $d_S(c',c'')\leq \delta$ and $d_S(c'',\{ a,b \})\geq r+1$. In particular, $c''\in \pi'\cap N_{2\delta}(c)$.
\end{proof}

The following theorem is a statement of \cite[Theorems 1.1 and 1.2]{Bow08} with effective bounds given. We refer to the $| G(a,b) \cap N_{\delta}(c)|$ and $|G(a,b;r)\cap N_{2\delta}(c)|$ as the `slices'.

\begin{thm}
\label{slices}

Fix $\delta\geq 3$ such that $\mathcal{C}^1(S)$ is $\delta$-hyperbolic for all surfaces $S$ with $\xi(S)\geq 2$. Then the following statements hold, where $K$ is a uniform constant.

\begin{enumerate}

\item For any $a,b\in\mathcal{C}^0(S)$, for any curve $c\in\pi\in\mathcal{L}(a,b)$ we have $| G(a,b) \cap N_{\delta}(c)| \leq K^{\xi(S)}$.

\item For any $r\geq 0$, $a,b\in\mathcal{C}^0(S)$ such that $d_S(a,b)\geq 2r+2k+1$ (where $k=10\delta +1$), then for any curve $c\in\pi\in\mathcal{L}(a,b)$ such that $c\notin N_{r+k}(a)\cup N_{r+k}(b)$ we have $|G(a,b;r)\cap N_{2\delta}(c)| \leq K^{\xi(S)}$.

\end{enumerate}

\end{thm}

\begin{proof}[Proof of \textnormal{(1)}]

We remind the reader of the notation of Theorem \ref{thm:tight}: $C(v_1,v_2;L)$ is the set of curves that lie on some tight filling multipath between $v_1,v_2\in\mathcal{C}^0(S)$ of length at most $L$; $B(S,L)$ is a computable bound for the cardinality for the set $C(v_1,v_2;L)$, which only depends on $S$ and $L$.

Given $c\in\pi\in\mathcal{L}(a,b)$, first suppose that $d_S(c,\{ a, b \})\geq 4\delta +1$. Let $c_a,c_b\in \pi$ be curves such that $d_S(c_a,c)=d_S(c,c_b)=4\delta+1$, $d_S(c_a,c)+d_S(c,c_b)=d_S(c_a,c_b)$ and $d_S(a,c_a)<d_S(a,c_b)$. We shall show that $G(a,b)\cap N_{\delta}(c) \subset C(c_a,c_b;12 \delta+2)$ to deduce that $|G(a,b)\cap N_{\delta}(c)|\leq B(S,12\delta +2)$.

Given any $v\in G(a,b)\cap N_\delta (c)$ we have $v\in \pi_T$, for some $\pi_T\in \mathcal{L}_T(a,b)$. By $\delta$-hyperbolicity, $d_S(c',\pi_T)\leq \delta$ whenever $c'\in \{c_a,c,c_b\}$ so by Theorem \ref{dir}, with $r=\delta$, there exists a tight filling multipath $P$ from $c_a$ to $c_b$ such that $\textnormal{length}(P)\leq 12\delta+2$ and $v\in \pi_T\cap N_\delta(c) \subset P\cap N_\delta (c)$. Thus $v\in C(c_a,c_b;12\delta +2)$ and the theorem is proved.

If $d_S(c,b)< 4\delta +1$ and $d_S(a,c)\geq 4\delta +1$ then we make an analogous argument with $c_a,c$ and $b$ and use Theorem \ref{dir}. If $d_S(a,c),d_S(c,b)< 4\delta +1$ then since every tight multigeodesic between $a$ and $b$ in this case is a tight filling multipath of length at most $8\delta +2$. Therefore it has length at most $12\delta+2$ thus $|G(a,b)|\leq B(S,12\delta +2)$. By Theorem \ref{thm:unifhyp} and Theorem \ref{thm:tight} we are done. \end{proof}

\begin{proof}[Proof of \textnormal{(2)}]

Given $c\in\pi\in\mathcal{L}(a,b)$ with $d_S(c,\{ a,b\} )\geq r+10\delta +2$, fix curves $c_a,c_b\in\pi$ such that $d_S(c_a,c)=d_S(c,c_b)=8\delta +1$, $d_S(c_a,c)+d_S(c,c_b)=d_S(c_a,c_b)$ and $d_S(a,c_a)<d_S(a,c_b)$. Then $d_S(\{c_a, c_b \} , \{ a, b\} ) \geq r + 2\delta +1$.

Given any $v\in G(a,b;r)\cap N_{2\delta}(c)$, we have $v\in\pi_T$, for some $\pi_T\in\mathcal{L}_T(a,b;r)$. By Lemma \ref{fellow}, $d_S(c',\pi_T)\leq 2\delta$ whenever $c'\in\{ c_a, c, c_b\}$. So by Theorem \ref{dir}, with $r=2\delta$, there exists a tight filling multipath $P$ of length at most $24\delta+2$ from $c_a$ to $c_b$ with $v\in \pi_T\cap N_{2\delta}(c)\subset P\cap N_{2\delta}(c)$. We conclude that $G(a,b;r)\cap N_{2\delta}(c)\subset C(c_a,c_b;24\delta+2)$ and thus $|G(a,b;r)\cap N_{2\delta}(c)| \leq B(S,24\delta+2)$. By Theorem \ref{thm:unifhyp} and Theorem \ref{thm:tight} we are done.\end{proof}

As a corollary of \cite[Section 3]{Bow08} using our bounds from Theorem \ref{slices} we have

\begin{thm}
\label{denominator}
There exists a computable $m=m(\xi(S))$ such that whenever $\phi$ is pseudo-Anosov, then the mapping class $\phi^m$ preserves a geodesic in $\mathcal{C}^1(S)$. Furthermore, $m$ is bounded by $\exp(\exp({K'\xi(S)}))$, where $K'$ is a uniform constant.
\end{thm}

\begin{proof}

In \cite[Lemma 3.4]{Bow08} for any $\phi$, there exists $m'\leq P^2$ such that $\phi^{m'}$ preserves a geodesic in $\mathcal{C}^0(S)$, where $P$ represents the bounds on slices as in Theorem \ref{slices}. Thus, we may take $m=\textnormal{lcm}(1,...,P^2)$ so that $\phi^m$ also preserves a geodesic, since $m'$ divides $m$. We can bound $m$ by $P^{2\pi(P^2)}$, where $\pi$ is the prime-counting function. We have $\pi(P^2)\leq \frac{2P^2}{\textnormal{log}(P^2)}$, see \cite{RossScho}. Now $P$ is bounded by $K^{\xi(S)}$ by Theorem \ref{slices} and we are done.\end{proof}

\subsection{Lower bounds on slices} \label{lowerbound}

We now give examples to show that any upper bounds in Theorem \ref{slices} must diverge exponentially with respect to $\xi(S)$, so in a sense the divergent behaviour of our upper bounds is best possible. We don't expect our examples here to have the largest possible slices.

Note that for $g\geq 1$, our examples given below are pairs of non-separating curves, and therefore these pairs of curves lie on bi-infinite geodesics in the curve graph. To see this, one can use high powers of pseudo-Anosovs on the complements of the curves and use the bounded geodesic image theorem of Masur--Minsky \cite{MasurMinsky00}, \cite{Webb13} to extend the geodesic to a longer one. For full details of this argument see \cite{BirmanMenasco12}. The point here is that high cardinality of slices occurs in long geodesics and not just short ones.

First, as an illustrative example, we give a pair of filling curves on $S_{5,10}$. See Figure \ref{fig:example}, the filling pair of curves is illustrated in black. We write $c_0$ to denote the horizontal black curve and $c_3$ to denote the other black curve. As illustrated, $c_3$ is a union of 5 arcs. The leftmost 2 arcs form an \textit{S shape}. The rightmost arc forms a \textit{snake}. This example can be generalized to $S_{2g+1,p}$ using $g$ S shapes and one snake. Similar examples can be made for the rest of surfaces with $\xi(S)\geq 2$.

We shall count possibilities for $c_2$ such that $c_0,c_1,c_2,c_3$ is a tight multigeodesic in order to establish a bound on slices from below.

In Figure \ref{fig:example} there is one light grey curve illustrated. It occurs as the boundary of the subsurface filled by $c_3$ and the dark grey curves, furthermore the dark grey curves occur as the boundary of the subsurface filled by the horizontal black curve and the light grey curve. Thus, the light grey curve serves as an example of $c_2$ and the dark grey curves serve as an example of $c_1$ such that $c_0,c_1,c_2,c_3$ is a tight multigeodesic.

Now we argue that there are more possibilities for $c_2$. The light grey curve illustrated envelops the left most S shape, and envelops the snake and punctures on the right, apart from two evenly spaced punctures that lie above $c_0$. There are $5$ punctures above $c_0$, we care about the middle and the rightmost punctures. The light grey curve illustrated represents a particular choice, one out of $2^{2+2}$---which out of the $2$ S shapes to envelop, and which of the $2$ punctures (middle or rightmost above $c_0$) to omit when enveloping the snake on the right. Out of these choices, one curve bounds a disc (as it envelops everything) and $2$ are peripheral (as it envelops everything except one puncture). Thus there are $2^{2+2}-1-2$ choices we could have made for the light grey curve. It is not difficult to see that all these choices lie on some tight geodesic between $c_0$ and $c_3$.

This example can be easily generalized to $S_{g,p}$, with $g=2g'+1$ and $p=4p'+2$, and the slice would have at least $2^{g'+p'}-1-p'$ curves. These examples can be generalized further to all surfaces with $\xi(S)\geq 2$. This demonstrates that any upper bound in Theorem \ref{slices} must be at least exponential with respect to the complexity of the surface.

\begin{figure}
\executeiffilenewer{example.svg}{example.pdf}%
{inkscape -z -D --file=example.svg %
--export-pdf=example.pdf --export-latex}%
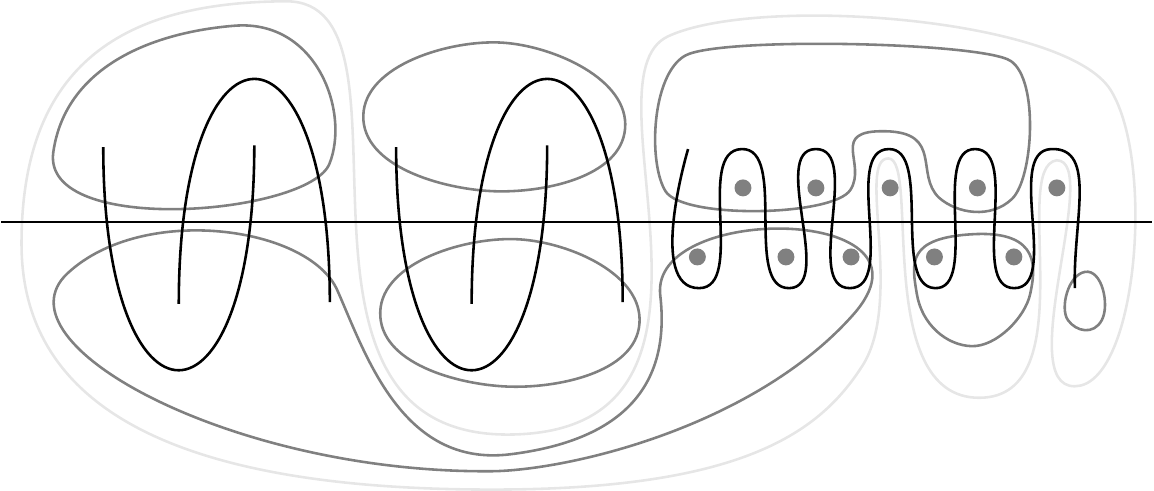%

\caption{Numbers represent handles and there is one point at infinity.}
\label{fig:example}
\end{figure}

\subsection{Acylindrical action constants}

The following definition is a remedy to having infinite vertex stabilizers in the case of the mapping class group acting on the curve graph. The definition says the `$r$-coarse stabilizers' of a sufficiently far away pair of points has uniform cardinality. 

\begin{defn} \label{def:acy}
A group $G$ acts on a path-metric space $(X,d_X)$ \textit{acylindrically} if for all $r\geq 0$, there exists $R,N$ such that whenever elements $a,b\in X$ satisfy $d_X(a,b)\geq R$ then there are at most $N$ elements $g\in G$ such that $d_X(a,ga),d_X(b,gb)\leq r$.
\end{defn}

We now give computable constants for the acylindrical action of the mapping class group on the curve graph. The main point is to say something about the behaviour of $R,N$ in terms of $r$ and $\xi(S)$. This may be of interest. The theorem was originally proved by Bowditch \cite{Bow08}, but without computable constants.

\begin{thm}\label{thm:acy}

The group $\mathcal{MCG}(S)$ acts on the metric space $\mathcal{C}^0(S)$ acylindrically. Fix $\delta,K$ as in Theorem \ref{slices}. We may take $R=4r+24\delta+7$ and $N=N_0(2r+4\delta+1)(8\delta+7)K^{2\xi(S)}$. In particular, fixing $r$, then $R$ is constant and $N$ grows exponentially with respect to $\xi(S)$. Fixing $\xi(S)$, then $R$ and $N$ grow linearly with respect to $r$.

\end{thm}

\begin{proof}

Suppose $d_S(a,b)\geq R$, then we may pick $x,y\in \pi \in \mathcal{L}_T(a,b)$ such that $d_S(x,y)=3$, and, $d_S(\{ x, y \} ,\{ a,b \} ) \geq r + (10\delta+1) + (2\delta+r) +1$.

Suppose $d_S(a,ga),d_S(b,gb)\leq r$. Let $x'$ and $y'$ be the nearest point projections of $gx$ and $gy$ onto $\pi$, respectively. We see that $d_S(x',\pi),d_S(y',\pi)\leq 2\delta$. Also, $d_S(x,x'),d_S(y,y')\leq r+2\delta$ and furthermore $d_S(x',y')\leq 4\delta+3$. Therefore $d_S(\{ x' , y' \}, \{ a, b \})\geq r + 10\delta + 2$ so by Theorem \ref{slices} there are at most $(2r+4\delta+1)K^{\xi(S)}$ possibilities for $gx$ and at most $(8\delta+7)K^{\xi(S)}$ possibilities for $gy$, given $gx$. By Lemma \ref{fin}, there are at most $N=N_0(2r+4\delta+1)(8\delta+7)K^{2\xi(S)}$ possibilities for $g$. By Theorem \ref{thm:unifhyp} we obtain the last statement.\end{proof}

In fact one can get $R=2r+13$ by using Proposition \ref{con}, Lemmas \ref{tight} and \ref{proc}, and some coarse geometric arguments. It is not known if we can take $R=2r+3$.

\subsection{Computing stable lengths of pseudo-Anosovs and invariant tight geodesic axes}

Now we outline a procedure to compute the stable length on the curve graph of any given pseudo-Anosov, answering a question of Bowditch asked at the `Aspects of hyperbolicity in geometry, topology, and dynamics' workshop in July 2011.

We won't make an effort here to optimize running times. We hope to return to this in a later paper along with a distance algorithm implied by Theorem \ref{thm:tight}.

\begin{thm} \label{stablecalc}
There exists a finite time algorithm that takes as input a surface $S$, a pseudo-Anosov $\phi$ on $S$ and returns the stable length of $\phi$.
\end{thm}

\begin{proof}
In Theorem \ref{denominator} we constructed a computable $m=m(\xi(S))$ such that $\phi^m$ preserves some geodesic axis. Then apply \cite[Proposition 7.1]{Shack12}.\end{proof}

Note that $\phi^{2m}$ preserves a geodesic with even translation length, thus by the tightening procedure a tight geodesic is preserved. The following proposition asserts that one can compute all invariant tight geodesics of any given pseudo-Anosov, returning an empty set if there are none.

\begin{thm} \label{tightcalc}
There exists a finite time algorithm that takes as input a surface $S$, a pseudo-Anosov $\phi$ on $S$ and returns all invariant tight geodesics of $\phi$. These are in the form of a collection of finite sets of curves each of whose orbit under $\phi$ is a tight geodesic.
\end{thm}

\begin{proof}

We start with calculating all invariant tight geodesics of $\phi^{2m}$ by considering $\psi=\phi^{2m(16\delta+15)}$, where $m$ is as in Theorem \ref{denominator}. Then we check whether or not under $\phi$ these are preserved.

Take any curve $c$ and calculate a geodesic $\pi$ from $c$ to $\psi(c)$. Calculate the sets $C'(a,b)=C(a,b;d_S(a,b)+4\delta)$ for $a,b\in \pi$, where we are using the notation of Theorem \ref{thm:tight}. Now we show that for all invariant tight geodesics $\pi_T$ of $\phi^{2m}$ there exists a sub-multipath of $\pi_T$ (whose orbit under $\phi^{2m}$ is $\pi_T$) that appears in some set $C'(a,b)$ with $a,b\in\pi$.

Let $\gamma\in \pi_T$ be a closest point projection of $c$ to $\pi_T$. By slimness with the rectangle $c,\psi(c),\psi(\gamma),\gamma$ the following statement holds: whenever $\alpha\in\pi_T$ is inbetween $\gamma$ and $\psi(\gamma)$ on $\pi_T$, and $d_S(\alpha,\{\gamma,\psi(\gamma)\})\geq 4\delta+1$, we have $d_S(\alpha,\pi)\leq 2\delta$. Therefore there exist $a,b\in\pi$ with $d_S(a,b)\geq ||\psi||-12\delta-2$, $d_S(a,\pi_T)\leq 2\delta$ and $d_S(b,\pi_T)\leq 2\delta$. By Proposition \ref{con} and Lemmas \ref{tight} and \ref{proc}, there exists a tight filling multipath from $a$ to $b$ of length at most $d_S(a,b)+4\delta$ with a sub-multipath of $\pi_T$ of length at least $d_S(a,b)-4\delta-12\geq ||\psi||-16\delta-14 = (16\delta+15)||\phi^{2m}||-16\delta-14\geq ||\phi^{2m}||$. Therefore $C'(a,b)$ contains the required sub-multipath of $\pi_T$.

The rest of the algorithm amounts to checking for each sub-multipath $P$ of $C'(a,b)$ if $P\cup \phi^{2m}(P)$ is a tight geodesic. Discard those that are not. Discard those $P$ that are not minimal in length. The length of any minimal one is equal to the stable length of $\phi^{2m}$. Now for each remaining $P$ check whether $P\cup \phi(P)$ is a subset of $P\cup \phi^{2m}(P)$. If it is then it follows that the orbit of $P$ under $\phi^{2m}$ is preserved under $\phi$. For those $P$ that remain, take any sub-multipath $P'$ of $P$ with $\frac{1}{2m}$ of the length. These are the required sets of curves. \end{proof}

\bibliography{tightreferences}
\bibliographystyle{amsalpha}

\end{document}

%% file: example.pdf_tex
\begingroup%
  \makeatletter%
  \providecommand\color[2][]{%
    \errmessage{(Inkscape) Color is used for the text in Inkscape, but the package 'color.sty' is not loaded}%
    \renewcommand\color[2][]{}%
  }%
  \providecommand\transparent[1]{%
    \errmessage{(Inkscape) Transparency is used (non-zero) for the text in Inkscape, but the package 'transparent.sty' is not loaded}%
    \renewcommand\transparent[1]{}%
  }%
  \providecommand\rotatebox[2]{#2}%
  \ifx\svgwidth\undefined%
    \setlength{\unitlength}{332.15786133bp}%
    \ifx\svgscale\undefined%
      \relax%
    \else%
      \setlength{\unitlength}{\unitlength * \real{\svgscale}}%
    \fi%
  \else%
    \setlength{\unitlength}{\svgwidth}%
  \fi%
  \global\let\svgwidth\undefined%
  \global\let\svgscale\undefined%
  \makeatother%
  \begin{picture}(1,0.42561616)%
    \put(0,0){\includegraphics[width=\unitlength]{example.pdf}}%
    \put(0.08349054,0.30956146){\color[rgb]{0,0,0}\makebox(0,0)[lb]{\smash{1}}}%
    \put(0.21562137,0.30956146){\color[rgb]{0,0,0}\makebox(0,0)[lb]{\smash{2}}}%
    \put(0.46971904,0.30956146){\color[rgb]{0,0,0}\makebox(0,0)[lb]{\smash{3}}}%
    \put(0.93274152,0.15597355){\color[rgb]{0,0,0}\makebox(0,0)[lb]{\smash{4}}}%
    \put(0.40421835,0.140163){\color[rgb]{0,0,0}\makebox(0,0)[lb]{\smash{5}}}%
    \put(0.33758828,0.30956146){\color[rgb]{0,0,0}\makebox(0,0)[lb]{\smash{2}}}%
    \put(0.58829801,0.30956146){\color[rgb]{0,0,0}\makebox(0,0)[lb]{\smash{3}}}%
    \put(0.53183182,0.14242166){\color[rgb]{0,0,0}\makebox(0,0)[lb]{\smash{4}}}%
    \put(0.28225148,0.140163){\color[rgb]{0,0,0}\makebox(0,0)[lb]{\smash{5}}}%
    \put(0.14899126,0.14355097){\color[rgb]{0,0,0}\makebox(0,0)[lb]{\smash{1}}}%
  \end{picture}%
\endgroup%